\documentclass[journal]{IEEEtran}
\usepackage[T1]{fontenc}

\IEEEoverridecommandlockouts                          

\usepackage{amsmath} 
\usepackage{amssymb}  
\usepackage{amsthm}
\usepackage{xcolor}
\usepackage{caption}
\usepackage{subcaption}
\usepackage{graphicx}
\usepackage{mathtools}
\usepackage{caption}
\usepackage{floatrow}
\usepackage{tikz}

\usepackage{pgfplots}
\usepackage{xcolor}
\usetikzlibrary{matrix,arrows,calc,positioning,shapes,decorations.pathreplacing}
\usepackage{graphicx}

\usepackage{amsmath} 

\usepackage{enumerate}
\usepackage[all,tips]{xy}
\SelectTips{cm}{11}

\newtheorem{theorem}{Theorem}

\newtheorem{definition}{Definition}

\newtheorem{lemma}{Lemma}

\usepackage{bm}

\newcommand{\R}{\mathbb{R}}

\newcommand{\D}{\mathcal{D}}

\newtheorem{rem}{Remark}
\newtheorem{exam}{Example}

\usepackage[noabbrev]{cleveref} 
\crefname{rem}{Remark}{Remarks}
\crefname{assum}{Assumption}{Assumptions}
\crefname{prop}{Proposition}{Propositions}
\crefname{propy}{Property}{Properties}
\crefname{cor}{Corollary}{Corollaries}
\crefname{lem}{Lemma}{Lemmas}
\crefname{section}{Section}{Sections}
\crefname{thm}{Theorem}{Theorems}
\crefname{figure}{Fig.}{Fig.}
\Crefname{figure}{Figure}{Figures}
\crefname{equation}{}{}

\usepackage{mathtools}
\usepackage{siunitx}

\begin{document}

\title{
Geometric stabilization of virtual linear nonholonomic constraints}

\author{Alexandre Anahory Simoes$^{3}$,  Anthony Bloch$^{2}$,  Leonardo Colombo$^{3}$ and Efstratios Stratoglou$^{4}$$^{*}$

\thanks{A. Anahory Simoes is with the School of Science and Technology, IE University, Spain. (e-mail: alexandre.anahory@ie.edu).}
\thanks{A. Bloch is with Department of Mathematics, University of Michigan, Ann Arbor, MI 48109, USA. (e-mail: abloch@umich.edu).}
\thanks{L. Colombo is with Centre for Automation and Robotics (CSIC-UPM), Ctra. M300 Campo Real, Km 0,200, Arganda del Rey - 28500 Madrid, Spain. (e-mail: leonardo.colombo@csic.es).}
\thanks{E. Stratoglou is with Universidad Polit\'ecnica de Madrid (UPM), Jos\'e Guti\'errez Abascal, 2, 28006 Madrid, Spain. (e-mail: ef.stratoglou@alumnos.upm.es).}
         
\thanks{$^*$ E. Stratoglou is the corresponding author.} 
\thanks{The authors acknowledge financial support from Grant PID2022-137909NB-C21 funded by MCIN/AEI/ 10.13039/501100011033. A.B. was partially supported by NSF grant  DMS-2103026, and AFOSR grants FA
9550-22-1-0215 and FA 9550-23-1-0400.}}

\maketitle

\begin{abstract}
 In this paper, we give sufficient conditions for and deduce a control law under which a mechanical control system converges exponentially fast to a virtual linear nonholonomic constraint that is control invariant via the same feedback control. Virtual constraints are relations imposed on a control system that become invariant via feedback control, as opposed to physical constraints acting on the system. Virtual nonholonomic constraints, similarly to mechanical nonholonomic constraints, are a class of virtual constraints that depend on velocities rather than only on the configurations of the system.
\end{abstract}


\IEEEpeerreviewmaketitle

\section{Introduction}
The concept of virtual holonomic constraints goes back to the work of \cite{Appell} and later i was examined more extensively in the thesis of \cite{Beghin:1922:thesis} as a  set of constraints that can be applied via the exertion of external forces. Virtual holonomic constraints have emerged as a useful tool for motion control in bipedal robots \cite{isidori:nonlinear}, \cite{westervelt2003hybrid}, \cite{westervelt2018feedback}. In \cite{CanudasdeWit:ontheconcept:2004} virtual holonomic constraints are used for constructing orbitally stable feedback laws for balancing and walking control. In subsequent years this approach was extended to motion planning for general robotic systems (see \cite{Consol:Constal:2015} for instance and references therein). 


In \cite{Maggiore:Consolini:2013}, the authors studied conditions under which holonomic constraints can be made invariant and stabilizable by feedback, and sufficient conditions for the constraint dynamics to coincide with an Euler-Lagrange system with holonomic constraints are provided. Next, in \cite{Consol:Constal:2015} the authors presented virtual holonomic constraints in an affine connection geometric framework, where the trajectories of the closed-loop virtual constrained system are described as geodesics of a connection obtained from the mechanical system's Levi-Civita connection. 

Virtual nonholonomic constraints are a class of virtual constraints that depend on velocities rather than only on the configurations of the system. Those constraints were introduced in \cite{griffin2015nonholonomic} to design a velocity-based swing foot placement in bipedal robots (see also \cite{moran2021energy}). 

In \cite{Simoes:linear:nonholonomic} we introduced the notion of virtual linear nonholonomic constraints in a geometric framework. More precisely, we introduce a controlled invariant distribution associated with an affine connection mechanical control system. We demonstrated the existence and uniqueness of a control law defining a virtual nonholonomic constraint and we characterized the trajectories of the closed-loop system as solutions of a mechanical system associated with an induced constrained connection. Moreover, we characterized when one can obtain nonholonomic dynamics from virtual nonholonomic constraints. We extended these results to affine and nonlinear virtual nonholonomic constraints in \cite{stratoglou2023bvirtual} and \cite{stratoglou2023geometry}, respectively, as well as for mechanical systems on Lie groups \cite{stratoglou2023virtual}. 

In this paper, we study the problem of stabilizing a system around  desired virtual nonholonomic constraints. In particular, we will prove that there is a  control law that drives the system into complying with the constraints. In addition, if the system is already complying with the constraints at some given point, then the control law coincides with the unique control law that ensures the existence of a virtual nonholonomic constraint as given in \cite{Simoes:linear:nonholonomic}.

The remainder of the paper is organized as follows. In Section \ref{sec2} we introduce nonholonomic constraints and nonholonomic systems in an affine connection framework. In Section \ref{sec3} we characterize virtual nonholonomic constraints as a controlled invariant distribution associated with an affine connection mechanical control system. We also give some examples and simulation results. We conclude the paper in Section \ref{conc} with some directions of future. 

\section{Nonholonomic mechanical systems}\label{sec2}
Let $Q$ be the configuration space of a mechanical system, a differentiable manifold with $\dim(Q)=n$, and with local coordinates denoted by $(q^i)$ for $i=1,\ldots,n$. Most nonholonomic systems have linear constraints in velocities, and these are the ones we will consider. Linear constraints on the
velocities (or Pfaffian constraints) are locally given by equations of the form \begin{equation}\label{NHconstraint}\phi(q^i, \dot{q}^i)=\mu_i(q)\dot{q}^i=0,\end{equation} depending in general, on the configurations and
velocities of the system (see \cite{bloch2003nonholonomic} for instance). From a geometric point of view, these constraints are defined by a regular distribution ${\mathcal D}$ on
$Q$ of constant rank $(n-m)$ such that the annihilator of ${\mathcal
D}$, denoted by $\mathcal {D}^o$, is locally given at each point of $Q$ by
${\mathcal D}^o_{q} = \operatorname{span}\left\{ \mu^{a}(q)=\mu_i^{a}dq^i \; ; 1 \leq a
\leq m \right\}$, where $\mu^{a}$ are linearly independent differential one-forms at each point of $Q$. We denote by $\Omega^{1}(Q)$ the set of differential one-forms on $Q$.

Next, consider  mechanical systems where the Lagrangian is of mechanical type, that is, mechanical systems with a dynamics described by a Lagrangian function $L:TQ\rightarrow\mathbb{R}$ which is defined by
\begin{equation}\label{mechanical:lagrangian}
    L(v_q)=\frac{1}{2}\mathcal{G}(v_q, v_q) - V(q),
\end{equation}
with $v_q\in T_qQ$, where $\mathcal{G}$ denotes a Riemannian metric on $Q$ representing the kinetic energy of the systems, $ T_qQ$, the tangent space at the point $q$ of $Q$, and
$V:Q\rightarrow\mathbb{R}$ is a (smooth) potential function, and also assume the Lagrangian system is subject to the nonholonomic constraints \eqref{NHconstraint}.

\begin{definition}\label{nonholonomicsystem}
A \textit{nonholonomic mechanical system} on a smooth manifold $Q$ is given
by the triple $(\mathcal{G}, V, \mathcal{D})$, where $\mathcal{G}$ is
a Riemannian metric on $Q,$ representing the kinetic energy of the
system, $V:Q\rightarrow\mathbb{R}$ is a smooth function representing the potential
energy, and $\mathcal{D}$ a regular distribution on $Q$
describing the nonholonomic constraints.
\end{definition}

Denote by $\tau_{\mathcal{D}}:\mathcal{D}\rightarrow Q$ the canonical
projection from $\mathcal{D}$ to $Q$, locally given by $\tau_{\mathcal{D}}(q^i, \dot{q}^i)=q^i$, and denote by 
$\Gamma(\tau_{\mathcal{D}})$ the set of sections of $\tau_{D}$, that is, $Z\in\Gamma(\tau_{\mathcal{D}})$ if $Z:Q\to\mathcal{D}$ satisfies $(\tau_{\mathcal{D}}\circ Z)(q)=q$. We also denote by $\mathfrak{X}(Q)$ the set of vector fields on $Q$. If $X, Y\in\mathfrak{X}(Q),$ then
$[X,Y]$ denotes the standard Lie bracket of vector fields.

In any Riemannian manifold, there is a unique connection $\nabla^{\mathcal{G}}:\mathfrak{X}(Q)\times \mathfrak{X}(Q) \rightarrow \mathfrak{X}(Q)$ called the \textit{Levi-Civita connection} satisfying the following two properties:
\begin{enumerate}
\item $[ X,Y]=\nabla_{X}^{\mathcal{G}}Y-\nabla_{Y}^{\mathcal{G}}X$ (symmetry)
\item $X(\mathcal{G}(Y,Z))=\mathcal{G}(\nabla_{X}^{\mathcal{G}}(Y,Z))+\mathcal{G}(Y,\nabla_{X}^{\mathcal{G}}Z)$ (compatibillity of the metric).
\end{enumerate}
The trajectories $q:I\rightarrow Q$ of a mechanical Lagrangian determined by a Lagrangian function as in \eqref{mechanical:lagrangian} satisfy the following equation
\begin{equation}\label{ELeq}
    \nabla_{\dot{q}}^{\mathcal{G}}\dot{q} + \text{grad}_{\mathcal{G}}V(q(t)) = 0.
\end{equation}
Observe that if the potential function vanishes, then the trajectories of the mechanical system are just the geodesics with respect to the connection $\nabla^{\mathcal{G}}$. Here, the vector field $\text{grad}_{\mathcal{G}}V\in\mathfrak{X}(Q)$ is characterized by $$\mathcal{G}(\text{grad}_{\mathcal{G}}V, X) = dV(X), \; \mbox{ for  every } X \in
\mathfrak{X}(Q).$$

Using the Riemannian metric $\mathcal{G}$ we can define two
complementary orthogonal projectors ${\mathcal P}\colon TQ\to {\mathcal D}$ and ${\mathcal Q}\colon TQ\to {\mathcal
D}^{\perp},$ with respect to the tangent bundle orthogonal decomposition $\mathcal{D}\oplus\mathcal{D}^{\perp}=TQ$.

In the presence of a constraint distribution $\mathcal{D}$, equation \eqref{ELeq} must be slightly modified as follows. Consider the \textit{nonholonomic connection} $\nabla^{nh}:\mathfrak{X}(Q)\times \mathfrak{X}(Q) \rightarrow \mathfrak{X}(Q)$ defined by (see \cite{bullo} for instance)
\begin{equation}\label{nh:connection}
    \nabla^{nh}_X Y =\nabla_{X}^{\mathcal{G}} Y + (\nabla_{X}^{\mathcal{G}} \mathcal{Q})(Y).
\end{equation}
Then, the trajectories for the nonholonomic mechanical system associated with the Lagrangian \eqref{mechanical:lagrangian} and the distribution $\mathcal{D}$ must satisfy the following equation
\begin{equation}\label{nonholonomic:mechanical:equation}
    \nabla^{nh}_{\dot{q}}\dot{q} + \mathcal{P}(\text{grad}_{\mathcal{G}}V(q(t))) = 0.
\end{equation}

\section{Virtual nonholonomic constraints}\label{sec3}

Given an external force $F^{0}:TQ\rightarrow T^{*}Q$ and a control force $F:TQ\times U \rightarrow T^{*}Q$ of the form
\begin{equation}
    F(q,\dot{q},u) = \sum_{a=1}^{m} u_{a}f^{a}(q)
\end{equation}
where $f^{a}\in \Omega^{1}(Q)$ with $m<n$, $U\subset\mathbb{R}^{m}$ the set of controls and $u_a\in\mathbb{R}$ with $1\leq a\leq m$ the control inputs, consider the associated mechanical control system of the form
\begin{equation}\label{mechanical:control:system}
    \nabla^{\mathcal{G}}_{\dot{q}(t)} \dot{q}(t) =Y^{0}(q(t),\dot{q}(t)) + u_{a}(t)Y^{a}(q(t)),
\end{equation}
with $Y^{0}=\sharp(F^{0})$ and the point-wise linearly independent vector fields $Y^{a}=\sharp(f^{a})$ the corresponding force vector fields.

Hence, $q$ is the trajectory of a vector field of the form
\begin{equation}\label{SODE}\Gamma(v_{q})=G(v_{q})+u_{a}(Y^{a})_{v_{q}}^{V},\end{equation}
where $G$ is the vector field determined by the unactuated forced mechanical system
\begin{equation*}
    \nabla^{\mathcal{G}}_{\dot{q}(t)} \dot{q}(t) =Y^{0}(q(t),\dot{q}(t))
\end{equation*}
and where the vertical lift of a vector field $X\in \mathfrak{X}(Q)$ to $TQ$ is defined by $$X_{v_{q}}^{V}=\left. \frac{d}{dt}\right|_{t=0} (v_{q} + t X(q)).$$

\begin{definition}
    The distribution $\mathcal{F}\subseteq TQ$ generated by the vector fields  $\sharp(f_{i})$ is called the \textit{input distribution} associated with the mechanical control system \eqref{mechanical:control:system}.
\end{definition}

Now we will define the concept of virtual nonholonomic constraint.

\begin{definition}
A \textit{virtual nonholonomic constraint} associated with the mechanical control system \eqref{mechanical:control:system} is a controlled invariant distribution $\mathcal{D}\subseteq TQ$ for that system, that is, 
there exists a control function $\hat{u}:\mathcal{D}\rightarrow \mathbb{R}^{m}$ such that the solution of the closed-loop system satisfies $\phi_{t}(\mathcal{D})\subseteq \mathcal{D}$, where $\phi_{t}:TQ\rightarrow TQ$ denotes its flow.
\end{definition}

\begin{rem}A particular example of mechanical control system appearing in applications is determined by a mechanical Lagrangian function $L:TQ\rightarrow \mathbb{R}$. In this case, the control system is given by the controlled Euler-Lagrange equations, i.e., 
\begin{equation}\label{euler:lagrange:system}
    \frac{d}{dt}\left(\frac{\partial L}{\partial \dot{q}}\right)-\frac{\partial L}{\partial q}=F(q,\dot{q},u).
\end{equation}

If the curve $q:I\rightarrow Q$ is a solution of the controlled Euler Lagrange equations \eqref{euler:lagrange:system}, it may be shown that it satisfies the mechanical equation (see \cite{bullo} for instance)
\begin{equation}\label{lagrangian:control:system}
    \nabla^{\mathcal{G}}_{\dot{q}(t)} \dot{q}(t) + \text{grad}_{\mathcal{G}}V(q(t))=u_{a}(t)Y^{a}(q(t)).
\end{equation}
These are the equations of a mechanical control system as in \eqref{mechanical:control:system}, where the force field $Y^{0}$ is simply given by $-\text{grad}_{\mathcal{G}}V(q(t))$. In this case, we call \eqref{lagrangian:control:system} a controlled Lagrangian system.\hfill$\diamond$\end{rem}

The following theorem, given in \cite{Simoes:linear:nonholonomic}, establishes the conditions under which one can guarantee the existence and uniqueness of a control law making the distribution a virtual constraint.

\begin{theorem}\label{existence:theorem}
If the distribution $\mathcal{D}$ and the control input distribution $\mathcal{F}$ are transversal, then there exists a unique control function $u:\D\to\R^{m}$ making the distribution a virtual nonholonomic constraint associated with the mechanical control system \eqref{mechanical:control:system}.
\end{theorem}

\section{Stabilization of virtual nonholonomic constraints}

In this Section we describe the main contributions of the paper. In particular, we will prove that there is a more general control law that drives the system into complying with the constraint. In addition, if the system  already complies with the constraint at some given point, we will show that then the control law coincides with the one given in Theorem \ref{existence:theorem}.

Suppose that $\D$ is a distribution on $Q$ defined as the set of points where
$$\mu^{b}_{i}(q) \dot{q}^{i} = 0, \text{ for } b=1,\ldots, m.$$
We will denote by $\hat{\mu}^{b}$ the fiberwise linear function $\hat{\mu}^{b}:TQ\to \R$ defined by $\hat{\mu}^{b}(q, \dot{q}) = \mu^{b}_{i}(q) \dot{q}^{i}$.

Before stating the main theorem, we will need a Lemma describing the interplay between the distribution $\D$ and a distribution $\mathcal{F}$ that is transversal to it.

\begin{lemma}\label{invertibility:lemma}
    Let $\D$ be the distribution defined by the one-forms $\{\mu^{b}\}$ and let $\mathcal{F}$ be a transversal distribution generated by the vector  fields $\{Y^{b}\}$. The matrix
    $$C^{ab}=(Y^{a})^{V}(\hat{\mu}^{b})$$
    is invertible with smooth inverse.
\end{lemma}
\begin{proof}
    It is easy to see that
    $$(Y^{a})^{V}(\hat{\mu}^{b}) = \mu^{b}(Y^{a}).$$
    $C^{ab}$ has full rank, since its columns are linearly independent. In fact suppose that
    \begin{equation*}
        c_{1}\begin{bmatrix} \mu^{1}(Y^{1}) \\
        \vdots \\
        \mu^{m}(Y^{1}) \end{bmatrix} + \cdots + c_{m}\begin{bmatrix} \mu^{1}(Y^{m}) \\
        \vdots \\
        \mu^{m}(Y^{m}) \end{bmatrix}= 0,
    \end{equation*}
    which is equivalent to
    \begin{equation*}
        \begin{bmatrix} \mu^{1}(c_{1}Y^{1}+\cdots + c_{m}Y^{m}) \\
        \vdots \\
        \mu^{m}(c_{1}Y^{1}+\cdots + c_{m}Y^{m}) \end{bmatrix}=0.
    \end{equation*}
    By transversality, we have $\mathcal{D}\cap \mathcal{F} = \{0\}$ which implies that $c_{1}Y^{1}+\cdots + c_{m}Y^{m}=0$. Since $\{Y_{i}\}$ are linearly independent we conclude that $c_{1}=\cdots=c_{m}=0$ and $C^{ab}$ has full rank. But, since $C^{ab}$ is an $m\times m$ matrix, and $\mathcal{D}$ is a regular distribution, it must be invertible and its inverse must be smooth.
\end{proof}


\begin{theorem}\label{main:theorem}
    Given a mechanical control system of the form \eqref{mechanical:control:system} and a virtual constraint determined by the distribution $\D$ defined by the annihilation of the fiberwise linear functions $\hat{\mu}^{b}:TQ\rightarrow \R$. Suppose $\mathcal{F}$ is a transversal distribution to $\D$ generated by the vector field $\{Y^{a}\}$. 
    
    Then, the control law $u^{*}:TQ \to \R^{m}$ given by the expression
    $$u_{a}^{*} = C_{a b}(-\hat{\mu}^{b}-G(\hat{\mu}^{b})),$$
    where $G$ is the geodesic vector field and $C_{ab}$ is the inverse matrix of $C^{ab}=\mu^{b}(Y^{a})$, satisfies
    \begin{enumerate}
        \item $u^{*}|_{\D}$ is the unique control law whose existence is guaranteed by Theorem \ref{existence:theorem}.
        \item $\hat{\mu}^{b}\to 0$ exponentially fast along the system trajectories, for $b=1,\ldots, m$.
    \end{enumerate}
\end{theorem}
\begin{proof}
    The proof is constructive. We will construct a control law satisfying the previous requirements. Notice that if $q$ is a trajectory of the controlled system \eqref{mechanical:control:system}, then it is a trajectory of a vector field $\Gamma$ with the form \eqref{SODE}, where $G$ is now the geodesic vector field associated to the Riemannian metric $\mathcal{G}$.

    We would like to determine if there exists a control law $u^{*}$ for which the associated closed-loop system satisfies
    $$\Gamma(\hat{\mu}^{b}) = - \hat{\mu}^{b}, \text{ for } b=1, \ldots, m.$$

    This is equivalent to the equations
    $$u_{a}\left[(Y^{a})^{V}(\hat{\mu}^{b})\right] = - \hat{\mu}^{b} - G(\hat{\mu}^{b}).$$
    We remark that the matrix $C^{ab}=(Y^{a})^{V}(\hat{\mu}^{b})=\mu^{b}(Y^{a})$ is invertible by Lemma \ref{invertibility:lemma}. Therefore, we may solve the previous equation for the controls obtaining
    $$u_{a}^{*} = C_{ab}(- \hat{\mu}^{b} - G(\hat{\mu}^{b})),$$
    where $C_{ab}$ is the inverse of $C^{ab}$. Finally, the closed-loop system fulfills the following two requirements:
    \begin{enumerate}
        \item At points $v_{q}\in \D$, the closed-loop system is tangent to $\D$ since
        \begin{equation*}
            \begin{split}
                \Gamma(v_{q})(\hat{\mu}^{b}) & = G(v_{q})(\hat{\mu}^{b}) + u_{a}^{*}(v_{q})(Y^{a})^{V}(v_{q})(\hat{\mu}^{b}) \\
                & = G(v_{q})(\hat{\mu}^{b}) + u_{a}^{*}(v_{q})\mu^{b}(Y^{a}) \\
                & = G(v_{q})(\hat{\mu}^{b}) + C^{ab}u_{a}^{*}(v_{q}).
            \end{split}
        \end{equation*}
        Inserting the value of $u^{*}_{a}$, we deduce
        \begin{equation*}
            \begin{split}
                \Gamma(v_{q})(\hat{\mu}^{b}) & = G(v_{q})(\hat{\mu}^{b}) + (- \hat{\mu}^{b}(v_{q}) - G(v_{q})(\hat{\mu}^{b})) \\
                & = - \hat{\mu}^{b}(v_{q}).
            \end{split}
        \end{equation*}
        But, the right-hand side of the last equation vanishes because at $v_{q}\in\D$, we have that $\mu^{b}(v_{q})=0$ for $b=1, \ldots, m$. Thus, $\Gamma$ must be tangent to $\D$. By uniqueness of the control law given in Theorem $\ref{existence:theorem}$, we must have that $u^{*}|_{\D}$ is the unique control law making $\D$ a virtual constraint.
        \item By construction $\Gamma(\hat{\mu}^{b}) = - \hat{\mu}^{b}, \text{ for } b=1, \ldots, m.$ Consequently, if $c:[a,b]\to Q$ is a trajectory of the closed-loop system, then the function $h^{b}(t):=\hat{\mu}^{b}(c(t),\dot{c}(t))$ satisfies the differential equation $\dot{h}^{b}=-h^{b}$, which implies that $h^{b}(t)=h^{b}(0)e^{-t}$ for each $t\in [a,b]$ and $b=1,\ldots,m$. Therefore, each $h^{b}$ converges exponentially fast to $0$.
    \end{enumerate}
\end{proof}

\begin{rem}
    In the case that $\dim{\mathcal{F}}=1$, the expression for the control law $u^{*}$
    simplifies to
    $$u^{*} = - \frac{\hat{\mu}}{\mu(Y)} - \frac{G(\hat{\mu})}{\mu(Y)}.$$
\end{rem}

\begin{rem}
     Maggiore and Consolini \cite{Maggiore:Consolini:2013} provide a stabilization control law using feedback linearization in the case of 1-dimensional virtual holonomic constraint, i.e., in the case that the constraint manifold is a curve on the configuration space. We may deduce a similar formula in the case that our virtual constraint $\D$ is a one-dimensional distribution and, therefore, it is integrable.
     
     Let us suppose, for convenience, that $(q_{1}, \ldots, q_{n})$ is a local coordinate chart on the configuration manifold $Q$ for which the virtual holonomic constraint is the tangent space of the curve defined by
     $$\gamma(q_{n})=(\phi_{1}(q_{n}), \ldots, \phi_{n-1}(q_{n}), q_{n}),$$
     where $\phi_{i}:I_{i}\subseteq \R\to \R$ for each $i=1,\ldots, n-1$ and $I_{i}$ is an open interval. Then, $\D$ is the distribution spanned by the tangent vector
     $$\D_{\gamma(q_{n})} = \left\langle (\phi_{1}^{'}(q_{n}), \ldots, \phi_{n-1}^{'}(q_{n}), 1) \right\rangle.$$
     Therefore, $\D_{\gamma(q_{n})}$ is the set of vectors $(v_{1}, \ldots, v_{n})\in T_{\gamma(q_{n})} Q$ defined by the $n-1$ constraints
     $$\phi_{1}^{'}(q_{n})dq^{n}-dq^{1} = 0, \ldots, \phi_{n-1}^{'}(q_{n})dq^{n}-dq^{n-1} = 0.$$
     Let $\mu^{i} = \phi_{i}^{'}(q_{n})dq^{n}-dq^{i}$ for each $i=1,\ldots, n-1$.

     Next, suppose we have a mechanical control system of type \eqref{mechanical:control:system}, where the input distribution $\mathcal{F}$ and $\D$ satisfy the transversality assumption. Then, 
     $$C^{ab}=\phi_{b}^{'}(q_{n})Y^{a}_{n} - Y^{a}_{b}$$
     and $\hat{\mu}^{b}= \phi_{b}^{'}(q_{n})\dot{q}^{n} -  \dot{q}^{b}$. Hence,
     $$u_{a}^{*} = (\phi_{b}^{'}(q_{n})Y^{a}_{n} - Y^{a}_{b})^{-1} (- \phi_{b}^{'}(q_{n})\dot{q}^{n} + \dot{q}^{b} - G(\phi_{b}^{'}(q_{n}) \dot{q}^{n} - \dot{q}^{b})).$$
     Now, it is a direct computation to show that our control law coincides with the control law appearing in \cite{Maggiore:Consolini:2013} by considering the explicit form of the geodesic vector field $G$ that they have and setting $k_{1}=0$ and $k_{2}=1$ in their formula.
     
     We also remark that, the fact that our control law only coincides with the one in \cite{Maggiore:Consolini:2013} after $k_{1}=0$ is related with our original goal. We are not stabilizing the configuration space constraints but only the velocity (nonholonomic) constraints.
\end{rem}


\begin{exam}\label{nh:robot}
    Consider the nonholonomic system in $SE(2)$ whose Lagrangian function is $$L(x,y,\theta,\dot{x},\dot{y},\dot{\theta})=\frac{m}{2}(\dot{x}^{2}+\dot{y}^{2})+\frac{I\dot{\theta}^{2}}{2}$$ and the control force $$F(x,y,\theta,\dot{x},\dot{y},\dot{\theta},u)=u(\sin \theta dx-\cos \theta dy).$$
    The corresponding controlled Lagrangian system is
    \begin{equation*}
        m\ddot{x}=u \sin\theta, \quad m\ddot{y}=-u \cos\theta, \quad I\ddot{\theta}=0.
    \end{equation*}
    
    The input distribution $\mathcal{F}$ is generated just by one vector field $$Y=\frac{\sin \theta}{m}\frac{\partial}{\partial x}-\frac{\cos \theta}{m}\frac{\partial}{\partial y},$$
    while the virtual nonholonomic constraint is the distribution $\mathcal{D}$ determined by $\dot{x}\sin \theta - \dot{y}\cos \theta = 0$.
    
    On one hand, the control law
    $$\hat{u}(x,y,\theta,\dot{x},\dot{y},\dot{\theta})=-m\dot{\theta}(\cos\theta \dot{x} +\sin \theta \dot{y})$$
    is the unique function from Theorem \ref{existence:theorem} that makes the distribution invariant under the closed-loop system. Note that the nonholonomic system under study is the Chaplygin sleigh.
   
    On the other hand, the control law 
    $$u^{*}(x,y,\theta,\dot{x},\dot{y},\dot{\theta})=-m\dot{\theta}(\cos\theta \dot{x} +\sin \theta \dot{y}) - m (\dot{x}\sin \theta - \dot{y} \cos \theta)$$
    coincides with $\hat{u}$ on values of $\D$ and the constraint function $\hat{\mu}=\dot{x}\sin\theta - \dot{y}\cos\theta$ converges exponentially fast to zero. Indeed, differentiating the constraint function $\hat{\mu}$ along trajectories of the controlled system and using the dynamics, we deduce that
    $$\dot{\hat{\mu}} = \frac{u}{m} + \dot{x}\dot{\theta} \cos \theta + \dot{y}\dot{\theta}\sin \theta.$$

    Now, substituting $u$ by the control law $u^{*}$, we obtain that $\dot{\hat{\mu}}=-\hat{\mu}$ and $\hat{\mu}$ converges exponential fast to $0$ along trajectories of the closed-loop system. In Figure \ref{fig:1} we plot the projection of a trajectory of the closed-loop system into the plane $xy$  and, in Figure \ref{fig:2}, we plot the value of the constraint function $\hat{\mu}(t)$ along the same trajectory. We have considered the system's mass $m=2$ and its moment of inertia $I=1.5$. The total simulation time was $50$ seconds, with a time step of $0.01$ seconds, resulting in $5000$ steps. The initial conditions were $x_0 = 1, y_0 = 1, \theta_0 = \pi$ for position, and $\dot{x}_{0} = 0.5, \dot{y}_{0} = 8, \dot{\theta}_{0} = 0.1$ for velocity.


    \begin{figure}[htb!]
        \centering
        \includegraphics[width=0.9\linewidth]{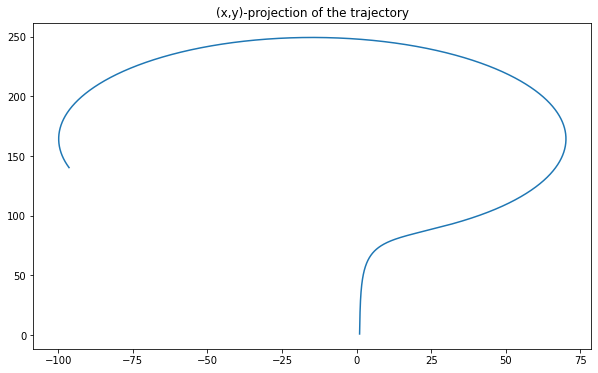}
        \caption{The projection of a trajectory of the closed-loop system into the plane $xy$ in Example \ref{nh:robot}.}
        \label{fig:1}
    \end{figure}

    \begin{figure}[htb!]
        \centering
        \includegraphics[width=0.9\linewidth]{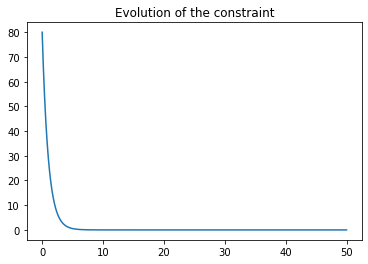}
        \caption{The constraint function $\hat{\mu}(t)$ along the same trajectory in Example \ref{nh:robot}.}
        \label{fig:2}
    \end{figure}
\end{exam}

\begin{exam}
   In this example we consider a nonholonomic system as in Example \ref{nh:robot} but with non-orthogonal control force. Consider in $SE(2)\cong \mathbb{R}^{2}\times \mathbb{S}^{1}$ the controlled mechanical system whose Lagrangian function is the same as Example \ref{nh:robot} but with the control force
    $$F(x,y,\theta,\dot{x},\dot{y},\dot{\theta},u)=u(\sin \theta dx-\cos \theta dy +d\theta).$$
    That is, the corresponding force vector field is not orthogonal to the distribution $\D$. The corresponding controlled Lagrangian system is
\begin{equation*}
    m\ddot{x}=u \sin\theta, \quad m\ddot{y}=-u \cos\theta, \quad I\ddot{\theta}=u
\end{equation*}
and, as we will show, it has the following virtual nonholonomic constraint
$$\sin\theta \dot{x} - \cos\theta \dot{y}=0.$$
The input distribution $\mathcal{F}$ is generated just by one vector field $$Y=\frac{\sin \theta}{m}\frac{\partial}{\partial x}-\frac{\cos \theta}{m}\frac{\partial}{\partial y}+\frac{1}{I}\frac{\partial}{\partial \theta},$$
while the virtual nonholonomic constraint is the distribution $\mathcal{D}$ defined as the set of tangent vectors $v_{q}\in T_{q}Q$  where $\mu(q)(v)=0,$ with $\mu=\sin\theta dx - cos\theta dy$. Thus, we may write it as
$$\mathcal{D}=\hbox{span}\Big{\{} X_{1}=\cos \theta\frac{\partial}{\partial x} + \sin\theta \frac{\partial}{\partial y},\, X_{2}=\frac{\partial}{\partial \theta}\Big{ \}}.$$
The control law
$$\hat{u}(x,y,\theta,\dot{x},\dot{y},\dot{\theta})=-m\dot{\theta}(\cos\theta \dot{x} +\sin \theta \dot{y})$$
makes the distribution invariant under the closed-loop system.

In addition, the control law
$$u^{*}(x,y,\theta,\dot{x},\dot{y},\dot{\theta})=-m\dot{\theta}(\cos\theta \dot{x} +\sin \theta \dot{y}) - m (\dot{x}\sin \theta - \dot{y} \cos \theta)$$
    coincides with $\hat{u}$ on values of $\D$ and the constraint function $h=\dot{x}\sin\theta - \dot{y}\cos\theta$ converges exponentially fast to zero. The conclusion follows from the same argument as in the Example \ref{nh:robot}, since the dynamics on the variable $\theta$ does not intervene in the derivative of the constraint function $h$. Therefore, one can prove that along trajectories of the closed-loop system one has also $\dot{h}=-h$.
\end{exam}

\begin{exam}\label{disk:example}
    Consider the controlled vertical rolling coin whose configuration space is $\mathbb{R}^{2}\times \mathbb{S}^{1} \times \mathbb{S}^{1}$. The Lagrangian function is given by $$L(x,y,\theta,\varphi,\dot{x},\dot{y},\dot{\theta},\dot{\varphi})=\frac{m}{2}(\dot{x}^{2}+\dot{y}^{2})+\frac{I\dot{\theta}^{2}}{2} + \frac{J\dot{\varphi}^{2}}{2},$$
    where $m$ is the mass and $I, J$ are the moments of inertia. Consider also the control force \begin{align*}F(x,y,\theta,\varphi,\dot{x},\dot{y},\dot{\theta},\dot{\varphi},u)=&u_{1}(dx-\cos \varphi d\theta + d\varphi)\\ &+ u_{2}(dy-\sin \varphi d\theta + d\varphi).\end{align*}
The controlled Lagrangian system is then
\begin{equation*}
    m\ddot{x}=u_{1},\, m\ddot{y}=u_{2},\, I\ddot{\theta}=-u_{1}\cos \varphi - u_{2} \sin \varphi,\, J\ddot{\varphi}=u_{1} + u_{2}.
\end{equation*}
The virtual nonholonomic constraints associated to this system are defined by the following equations
\begin{equation*}
    \dot{x}=\dot{\theta}\cos \varphi, \quad \dot{y} = \dot{\theta}\sin\varphi.
\end{equation*}
Therefore, the input distribution $\mathcal{F}$ is the set
\begin{equation*}
    \begin{split}
        \mathcal{F}=\hbox{span}\Big{\{} Y^{1} = & \frac{1}{m}\frac{\partial}{\partial x}-\frac{\cos \varphi}{I}\frac{\partial}{\partial \theta}+\frac{1}{J}\frac{\partial}{\partial \varphi}, \\
        & Y^{2}= \frac{1}{m}\frac{\partial}{\partial y}-\frac{\sin \varphi}{I}\frac{\partial}{\partial \theta}+\frac{1}{J}\frac{\partial}{\partial \varphi} \Big{\}}, 
    \end{split}
\end{equation*} and the constraint distribution $\mathcal{D}$ is defined by the 1-forms $\mu^{1}=dx - \cos \varphi d\theta$ and $\mu^{2}=dy - \sin \varphi d\theta$, thus
$$\mathcal{D}=\Big{\{} X_{1}=\cos \varphi\frac{\partial}{\partial x} + \sin\varphi \frac{\partial}{\partial y} + \frac{\partial}{\partial \theta},\, X_{2}= \frac{\partial}{\partial \varphi}\Big{\}}.$$
We may verify that $\mathcal{D}$ is in fact controlled invariant under the control law
\begin{equation*}
        \hat{u}_{1}= -m\dot{\theta}\dot{\varphi}\sin \varphi, \quad
        \hat{u}_{2}=  m\dot{\theta}\dot{\varphi} \cos \varphi.
\end{equation*}

The matrix $C^{ab}$ appearing in Theorem \ref{main:theorem} is given by
$$C^{ab}=\begin{bmatrix}
    \frac{1}{m}+\frac{cos^{2} \varphi}{I} & \frac{\cos\varphi\sin\varphi}{I} \\
    \frac{\cos\varphi\sin\varphi}{I} & \frac{1}{m}+\frac{sin^{2} \varphi}{I}
\end{bmatrix}$$
and its inverse is
$$C_{ab}= \frac{m}{I+m}\begin{bmatrix}
    I+m \sin^{2} \varphi & -m \cos\varphi\sin\varphi \\
    -m \cos\varphi\sin\varphi & I+m \cos^{2} \varphi
\end{bmatrix}$$

The stabilizing control law from Theorem \ref{main:theorem} is then
$$\begin{bmatrix}
    u_{1}^{*} \\
    u_{2}^{*}
\end{bmatrix} = \frac{m}{I+m}\begin{bmatrix}
    I+m \sin^{2} \varphi & -m \cos\varphi\sin\varphi \\
    -m \cos\varphi\sin\varphi & I+m \cos^{2} \varphi
\end{bmatrix} \begin{bmatrix}
    c_{1} \\
    c_{2}
\end{bmatrix}$$
where
\begin{equation*}
    \begin{split}
        c_{1} & = -\dot{x}+\dot{\theta} \cos \varphi - \dot{\varphi}\dot{\theta}\sin\varphi
        \\
        c_{2} & = - \dot{y} + \dot{\theta} \sin \varphi + \dot{\varphi}\dot{\theta}\cos\varphi.
    \end{split}
\end{equation*}

Therefore, simplifying the matrix calculations
\begin{equation*}
    \begin{split}
        u_{1}^{*} & = \frac{m^{2}}{I+m} \left( \frac{I}{m}(-\dot{x}+\dot{\theta} \cos \varphi - \dot{\varphi}\dot{\theta}\sin\varphi) \right. \\
        & \left. + \dot{y}\cos\varphi\sin\varphi - \dot{\varphi}\dot{\theta}\sin\varphi -\dot{x}\sin^{2} \varphi \right)\\
        u_{2}^{*} & = \frac{m^{2}}{I+m}\left(\dot{x}\cos\varphi\sin\varphi + \dot{\varphi}\dot{\theta}\cos\varphi - \dot{y}\cos^{2} \varphi\right. \\
        & \left. + \frac{I}{m}(- \dot{y} + \dot{\theta} \sin \varphi + \dot{\varphi}\dot{\theta}\cos\varphi) \right).
    \end{split}
\end{equation*}

\begin{figure}[htb!]
        \centering
        \includegraphics[width=1\linewidth]{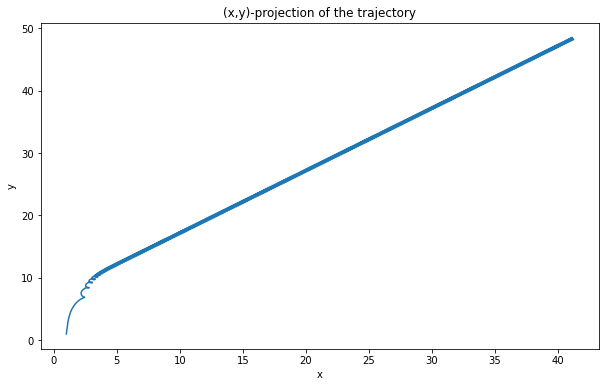}
        \caption{Projection of a trajectory of the closed-loop system into the plane $xy$. Example \ref{disk:example}.}
        \label{fig:3}
    \end{figure}

    \begin{figure}[htb!]
        \centering
        \begin{subfigure}[b]{0.49\textwidth}
         \centering
         \includegraphics[width=\textwidth]{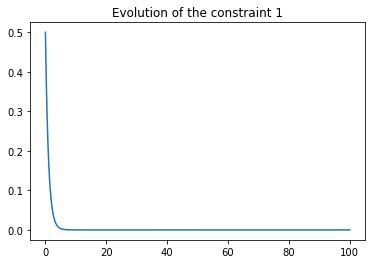}
         \caption{Constraint $\hat{\mu}_{1}(t)$.}
         \label{fig:4}
     \end{subfigure}
     \hfill
     \begin{subfigure}[b]{0.49\textwidth}
         \centering
         \includegraphics[width=\textwidth]{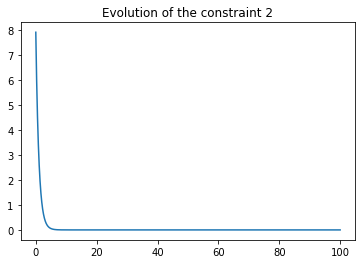}
         \caption{Constraint $\hat{\mu}_{2}(t)$.}
         \label{fig:5}
     \end{subfigure}\caption{Constraint functions in Example \ref{disk:example}.}
    \end{figure}

In Figure \ref{fig:3} we show the projection of a trajectory of the closed-loop system into the plane $xy$.
In Figures \ref{fig:4} and \ref{fig:5} we show the values of the constraint functions along the same trajectory. We have considered the system's mass $m=2$ and their moments of inertia $I=1.5$ and $J=1.1$. The total simulation time was again $100$ seconds with a time step of $0.01$ seconds implying a total of $10000$ steps. The initial conditions were $x_0 = 1, y_0 = 1, \theta_0 = \pi$ and $\varphi_{0}=\frac{\pi}{2}$ for position, and $\dot{x}_{0} = 0.5, \dot{y}_{0} = 8, \dot{\theta}_{0} = 0.1$ and $\dot{\varphi}_{0} = -0.1$ for velocity.

\end{exam}


\section{Conclusions and Future Work}\label{conc}
 In this paper, we investigated sufficient conditions and constructed a control law under which a mechanical control system converges exponentially fast to a virtual linear nonholonomic constraint that is control invariant via the same feedback control. We showed some examples and simulations to validate the mathematical result. 
 
 Some related questions to the main result of the paper have to be addressed in a future work. An interesting question is to understand the qualitative behavior of the energy. In particular, we have checked numerically that in general the energy stabilizes around a specific value or it converges to a bounded set of values. Understanding the limit value of the energy might give a clue on the nature of the closed-loop dynamics. In addition, we would like to characterize it geometrically as in \cite{stratoglou2023geometry}.

\bibliographystyle{IEEEtran}
\bibliography{autosam}

\end{document}